\documentclass[12pt,reqno]{amsart}
\usepackage{amsmath, amsthm, graphicx,amsfonts, amssymb,color}
\usepackage{bm}

\DeclareMathOperator*{\esssup}{ess\,sup}
\DeclareMathOperator*{\essinf}{ess\,inf}

\usepackage[backgroundcolor=white, bordercolor=blue,
linecolor=blue]{todonotes}
\newtheorem{thm}{Theorem}[section]

\newtheorem{lem}[thm]{Lemma}

\theoremstyle{definition}
\newtheorem{defn}[thm]{Definition}
\theoremstyle{remark}

\theoremstyle{example}

\numberwithin{equation}{section}
\usepackage{mathrsfs}

\newcommand{\R}{\mathbb R}
\newcommand{\eps}{\varepsilon}

\newcommand{\cL}{\mathcal{L}}

\renewcommand{\P}{\mathbb P}
\def\R{\mathbb R}

\def\E{\mathbb E}
\def\bg{\begin}
\def\be{\bg{equation}}
\def\de{\end{equation}}

\def\edar{\end{eqnarray}}
\newcommand{\nn}{\nonumber}
\def\lb{\label}

\def\l{\left}\def\r{\right}

\def\eps{\epsilon}

\def\q{\quad}

\def\ccap{\text {\rm cap}}

\def\[{\l[} \def\]{\r]}
\def\({\l(} \def\){\r)}

\def\bar{\overline}

\def\H{{\bf (Har)}}

\def\CG{{\bf (G)}}
\def\CGle{{\bf (G,\, \le)}}
\def\CGge{{\bf (G,\, \ge)}}
\def\CGad{{\bf (G^*)}}
\def\CGlead{{\bf (G^*,\, \le)}}
\def\CGgead{{\bf (G^*,\, \ge)}}
\def\CC{{\bf{(C)}}}
\def\CCle{{\bf (C,\, \le)}}
\def\CCge{{\bf (C,\, \ge)}}
\def\CE{{\bf (E)}}
\def\CEle{{\bf (E,\, \le)}}
\def\CEge{{\bf (E,\, \ge)}}
\def\CEad{{\bf (E^*)}}
\def\CElead{{\bf (E^*,\, \le)}}
\def\CEgead{{\bf (E^*,\, \ge)}}
\def\HG{{\bf (HG)}}

\def\1{{\bf 1}}

\newcommand{\ass}[1]{{\bf Assumption \,\,{#1}}}

\def\P{\mathbb P}

\def\a{\mathbf{a}}

\def\R{\mathbb{R}}
\def\Rd{\mathbb{R}^d}

\def\beqlb{\begin{eqnarray}}\def\eeqlb{\end{eqnarray}}
\def\beqnn{\begin{eqnarray*}}\def\eeqnn{\end{eqnarray*}}

\title{\bf  Capacity and Exit Time for Non-reversible Diffusions}
\author{Lu-Jing Huang,
	Kyung-Youn Kim*}

\begin{document}
\maketitle


\begin{abstract}
Capacity is an important quantity in potential theory and in the study of Markov processes. We give equivalent conditions between the capacity, the mean exit time, and the Green function for non-reversible diffusions.
\end{abstract}



\section{Introduction}

Potential theory and Markov processes are well connected and give the insight into each theory. For example, potential theory provides some efficient recurrence and transience criteria for Markov processes, see e.g. \cite{C04,DS84,GL14}. Specifically, the Dirichlet and Thompson principles, which express the capacity in terms of infimum and supremum respectively, can be used to prove the recurrence of Markov processes see \cite{BGL07,So94,Wo00} for more details.
Also capacities, Green functions, equilibrium potentials and equilibrium measures are basic tools to investigate the metastability of Markov processes, see e.g. \cite{BL12,B06,BEGK02,BEGK04} for reversible cases, and \cite{BL12a,L14,LMS19} for non-reversible cases.

Recently, the authors in \cite{GH14} established relation between the capacity, the mean exit time and the Green function for symmetric Markov processes, and find out the corresponding conditions are equivalent to the existence and estimates of the associated heat kernel. In this paper, we want to generalize these results to the non-reversible case.

The structure of this article is the following. In Section \ref{notat_main}, we introduce some necessary background material, such as assumptions of operators and processes, conditions on the Green function, the exit time and the capacity. We also present our main result in Theorem \ref{main}. In Section \ref{Green_cap}, we give some properties of the Green function and the capacity which are useful to prove our main result in Section 4. The proof of our main results are presented in Section \ref{proof}.

{\bf Notation}:
For any two nonnegative values $\Phi$ and $\Psi$, the notation $ \Phi\asymp \Psi$ means that there is a positive constant
$c>1$ such that $c^{-1}\ \Phi\le  \Psi\le c\ \Phi$.
The letter $c$ will denote a positive constant which depends only on non-essential parameters, and it may change at each appearance.


\section{Conditions and main result}\label{notat_main}

We start by introducing some notations.
For fixed $d\geq 3$, let $\Omega\subseteq \R^d$ be a bounded domain. Consider an operator $\cL$ in divergence form
\begin{equation}\label{gnrt}
\cL=\nabla\cdot \a\nabla+b\cdot\nabla,
\end{equation}
where $\a(x)=(a_{ij}(x))_{1\leq i,j\leq d}$ is a positive-definite matrix and $b(x)=(b_1 (x),\cdots,b_d (x))$ is a vector in $\Omega$.

Denote by $C^k(\Omega)$, $0\leq k\leq \infty$, the space of real functions on $\Omega$ whose partial derivatives up to order $k$ are continuous, and $C^{k,\alpha}(\Omega)$, $0<\alpha<1$, the subspace of $C^k(\Omega)$ containing functions whose $k$-th order partial derivatives are $\alpha$-H\"older continuous.
 We assume the following regularity conditions:
\begin{itemize}
\item[(A1)] $a_{ij},  b_{i}, \sum_{i=1}^d \partial_{x_i} a_{ij}\in C^{1, \alpha}(\Omega)$ for some $\alpha\in(0,1)$;\\

\item[(A2)] There exists a constant $\lambda>0$ such that $v\cdot  \a(x)v \geq \lambda|v|^2\text{ for all } x\in \Omega \text{ and } v\in\Omega\backslash \{0\}.$
\end{itemize}
By \cite[Theorem 1.13.1]{Pi95}, there exists a unique solution $\{\P_x:x\in\Omega\}$ to the martingale problem associated with $\cL$ satisfying  the strong Markov property and the Feller property.
Corresponding to operator $\cL$, there is a diffusion $X=(X_t)_{t\geq0}$ and a density $p(t,x,y)$ on $\Omega$, which is called heat kernel for $\cL$.
Throughout the paper, we assume that $X$ is positive recurrent with a unique invariant distribution $\mu(dx)$ satisfying that for any $t>0$ and any Borel set $B\subseteq \Omega$,
$$
\int_\Omega p(t,x,B)\mu(dx)=\mu(B).
$$
For the notational convenience, we use the same notation $\mu(x)$ to denote the density of $\mu$ with respect to Lebesgue measure.
Let $L^2(\mu)$ be the space of functions on $\Omega$ which are square-integrable with respect to $\mu$ with the inner product $\langle\cdot,\cdot\rangle_\mu$ defined by
$$
\langle f,g \rangle_\mu=\int_\Omega f(x)g(x)\mu(dx)=\int_\Omega f(x)g(x)\mu(x)dx.
$$

Let $D$ be a domain in $\Omega$ such that $\bar{D}\subset \Omega$, denoted by $D\Subset \Omega$. We assume that $D$ has a $C^{2,\alpha}-$boundary for some $\alpha\in(0,1)$, i.e., for each point $x\in\partial D$, there exist a ball $B$ centered at $x$ and a one-to-one map $\psi$ from $B$ to a set $C\subset \R^d$ such that
\begin{align*}
&\psi(B\cap D)\subset \{z\in \Rd:{z_d}>0\},\quad \psi(B\cap \partial D)\subset \{z\in \Rd: z_d=0\},\\
&\psi\in C^{k, \alpha}(B)\quad\quad\text {and}\quad\quad \psi^{-1}\in C^{k, \alpha}(C).
\end{align*}
For the first exit time $T_D=\inf\{t\geq0:X_t\notin D\}$ of $D$, consider the sub-diffusion $X^D$ which is killed upon exiting $D$, that is,
$$
X_t^D:=X_t\quad\text{ if }t< T_D,\q\text{ and }\q X_t:=X_{T_D}\quad\text{ if }t\ge T_D.
$$
Let $p_D(t,x,y)$ be the transition density of the killed process $X^D$. Since $X$ is positive recurrent in $\Omega$ and $D\Subset \Omega$, the killed diffusion $X^D$ is transient.
Thus \cite[Theorem 4.2.1 and 4.2.4]{Pi95} guarantee the existence of the Green operator $G_D$ satisfying
$$
G_Df(x)=\int_0^\infty\int_Dp_D(t,x,y)f(y)dydt\quad\text{for } f\in L^2(\mu),\ x\in\Omega.
$$
The kernel $g_D(x,y):=\int_0^\infty p_D(t,x,y)dt$ of $G_D$ is called the Green function satisfying that
\begin{equation}\label{g_e}
\E_x[T_D]=G_D{\bf 1}(x)=\int_D g_D(x,y)dy.
\end{equation}

For any domains $A,\ B\subset \Omega$, let us introduce the definition of the capacity for $X$ between $A, B$.
We first present the {\bf Assumption S} for $A$ and $B$ as follows:
\begin{itemize}
\item[(1)] $A,B\subset \Omega$ are two domains with $C^{2,\alpha}$-boundaries for some $\alpha\in(0,1)$;\\

\item[(2)] $\sigma(A),\ \sigma(B)<\infty$, where $\sigma$ is the measure on the boundary of set;\\

\item[(3)] $d(A,B):=\inf\{|x-y|:x\in A, y\in B\}>0$.
\end{itemize}

Consider the Poisson equation for $\cL$ on $A,B$:
\be\lb{pe}
\begin{cases}
\mathcal{L}u(x)=0, &x\in(\bar{A}\cup\bar{B})^c,\\
u(x)=1, &x\in \partial A,\\
u(x)=0, &x\in \partial B.
\end{cases}
\de
Under the conditions (A1)--(A2), there exists an unique solution
$$
h_{A,B}(x)=\P_x(\tau_A<\tau_B), \quad x\in\Omega,
$$
where $\tau_\cdot$ is the first hitting time of diffusion $X$, see e.g., \cite[Theorem 3.3.1]{Pi95}. Using the function $h_{A,B}$, we define the capacity between $A$ and $B$ by
$$
\ccap(A,B)=-\langle h_{A,B},\cL h_{A,B}\rangle_\mu.
$$

In order to present our main result, we introduce conditions of the Green function, the mean exit time and the capacity.

We say that the condition $\CG$ holds if there exist constants $ K, C\ge 1$ such that for any ball $B:=B(x_0, R)\subset \Omega$, the Green function $g_B(y, x_0)$ has the upper and lower bounds as follows:
		\begin{align*}
		{\CGle}:\qquad g_B(y,x_0)&\le C |x_0- y|^{2-d} \quad\quad\text{ for all}\,\,y\in B\backslash \{x_0\}\,, \\
		\CGge:\qquad g_B(y,x_0)&\ge C^{-1}|x_0- y|^{2-d} \quad\quad\text{for all}\,\, y\in  K^{-1}B\backslash \{x_0\}\,.
		\end{align*}
We say the condition $\CE$ holds if
there exist constants $C\ge 1, \delta\in(0, 1)$ such that for any ball $B\subset \Omega$ of radius $R>0$, the mean exit time $\E_{ \cdot}[T_B]$ has the upper and lower bounds as follows:
\begin{align*}
{\CEle}:&\qquad\esssup_{ x\in B} \E_x[T_B]\le C\, R^2 , \\
{\CEge}:&\qquad\essinf_{ x\in\delta  B} \E_x[T_B]\ge C^{-1}\,R^2\,.
\end{align*}
Finally, related to the capacity, we say condition $\CC$ is fulfilled if there exist constants $K, C\ge1$ such that for any ball $B$ of radius $R>0$ with $B, KB\subset \Omega$,
$\ccap(B, \overline {KB}^{\,c})$ has the following bounds:
		\begin{align*}
		\CC: \qquad C^{-1} \, \frac{\mu(B)}{R^2}\le \ccap(B, \overline {KB}^{\,c})\le C\,    \frac{\mu(B)}{R^2}\,.
		\end{align*}
Note that the constants $C, K, \delta$ of the conditions $\CG, \CE$ and $\CC$.

We denote by $X^*$ the dual process of $X$ with respect to $\mu$ with the generator
$$
\cL^*=\nabla\cdot \a\nabla+\big(\tfrac{2}{\mu}\a\nabla\mu-b\big)\cdot\nabla.
$$
Let $p^*(t,x,y)$ and $p^*_D(t,x,y)$ be the heat kernels of $X^*$ and $X^{*D}$ respectively.
Then we have the following relations between $p(t, x, y)$ ($p_D(t,x,y)$, respectively) and $p^*(t,x,y)$  ($p^*_D(t,x,y)$, respectively) with the invariant distribution $\mu$:
for all $t>0$, $x,y\in\Omega$,
\begin{equation}\label{dual:k}
\mu(x)p^*(t,x,y)=\mu(y)p(t,y,x)\quad
 \text{and}\quad
  \mu(x)p_D^*(t,x,y)=\mu(y)p_D(t,y,x).
\end{equation}
We will use $*$ all objects corresponding to the dual process $X^*$, e.g., $G^*$ and $g^*$ stand for its Green operator and Green function, and
$\CGad, \CEad$ and ${\bf (C^*)}$ the corresponding conditions.

Since $\mu\in C^1(\Omega)$ and $\Omega$ is bounded domain,
\eqref{dual:k} implies that
\begin{equation}\label{g_g*}
\CGle\Longleftrightarrow\CGlead\quad
 \text{and}\quad
  \CGge\Longleftrightarrow\CGgead.
\end{equation}

The following is our main result in this paper.

\begin{thm}\label{main}
Let $\cL$ be an operator defined in \eqref{gnrt} satisfying conditions (A1)--(A2). Denote by $X$ the corresponding diffusion and assume it is positive recurrent with invariant distribution $\mu$. Then we have
	\begin{align*}
	({\bf C})\Longleftrightarrow({\bf G})\Longleftrightarrow({\bf E^*}).
	\end{align*}
\end{thm}

 Here we want to point out that the constant $K>0$ in $\CC$ and $\CG$ is same,  and  $\delta=K^{-1}$ in $\CE$.

The characteristics related to the Green operators and the capacity will be introduced in Section 3, and we give the proof of Theorem \ref{main} in Section 4 by showing $\CC\Rightarrow \CG$, $\CG\Rightarrow\CEad$ and $\CEad\Rightarrow\CC$.

\section{Green operator and capacity}\label{Green_cap}

\subsection{Green operator and Green function}
In this subsection, we give some properties of the Green operator and the Green function which are useful to prove $\CC\Rightarrow\CGad$ in the subsection 4.1.

First, we introduce the definition of harmonic function.

\begin{defn}\label{d:har}
	For any open set $U\subset \Omega$, a function $u$ is called harmonic (subharmonic, superharmonic, respectively)
	for $\cL$ on $U$ if it satisfies
	$\mathcal{L}u=0\,(\geq,\,\leq \,,\text{respectively})$ in $U$.
	%
\end{defn}

The following lemma tells us that Green functions $g_\cdot(\cdot,y)$ and $g^*(\cdot,y)$ are harmonic.

\begin{lem}\label{Thm425}
Under the same conditions of Theorem \ref{main}, for any domain $D \Subset\Omega$, we have that
\begin{itemize}
\item[(1)] $g_D(x,y)$ and $g^*_D(x,y)$ are positive and jointly continuous on $U_1\times U_2$ whenever $U_1,U_2\subset D$ and $U_1\cap U_2=\emptyset$.
\item[(2)] For each $y\in D$, $g_D(\cdot, y)\in C^{2,\alpha}(D\backslash \{y\})$ and $\mathcal{L}g_D(\cdot,y)=0$ on $D\backslash \{y\}$.
\item[(3)] For each $y\in D$, $g_D^*(\cdot, y)\in C^{2,\alpha}(D\backslash \{y\})$ and $\mathcal{L}^*g_D^*(\cdot,y)=0$ on $D\backslash \{y\}$.
\end{itemize}
\end{lem}
\begin{proof}
Recall that \cite[Theorem 4.2.1 and 4.2.4]{Pi95} give the existences of Green operators $G_D$ and $G_D^*$ with kernels $g_D(x, y)$ and $g_D^*(x, y)$, respectively. Thus we obtain the desired results by \cite[Theorem 4.2.5]{Pi95}.
\end{proof}

From the following lemma, we can see that $G_D$ is considered as the inverse of $\cL$ in $D$.

\begin{lem}\label{LG=GL}
Let the conditions of Theorem \ref{main} hold. Assume $D \Subset\Omega$ with a $C^{2. \alpha}$-boundary for some $\alpha\in (0,1)$. Then for any $f\in \mathcal {B}^{\alpha}(D)\cap C^2(D)$, where $\mathcal {B}^{\alpha}(D):=\{u\in C^{\alpha}(\overline D):u=0\text{ on } \partial D\}$, we have
$$
\mathcal{L} G_D f=G_D \mathcal{L}  f=-f\quad \text{in}\ D.
$$
\end{lem}
	
\begin{proof}
From \cite[Theorem 3.6.4]{Pi95}, we see that $G_D f$ is the unique solution of
$$
\begin{cases}
\mathcal{L}u=-f,&\ \text{in}\ D;\\
u=0,&\ \text{on}\ \partial D,
\end{cases}
$$
that is, $\mathcal{L} G_D f=-f$ in $D$.
On the other hand, it is obvious that  $\mathcal{L}f$ belongs to $C^\alpha(\bar{D})$ and vanishes at $\partial D$.
Let $h=G_D\mathcal{L}f$. Then $h$ is the unique solution of
$$
\begin{cases}
\mathcal{L}u=-\mathcal{L}f,&\ \text{in}\ D;\\
u=0,&\ \text{on}\ \partial D.
\end{cases}
$$
Since $-f$ also solves the above equation, the uniqueness yields $h=-f$ which conclude our assertion.
\end{proof}

Nextly we will prove the maximum and minimum principles of the Green functions using the following lemma which is from \cite[Theorem 3.1]{GT01} or \cite[Theorem 3.2.1]{Pi95}.

\begin{lem}\label{prop4.2}
Let the conditions of Theorem \ref{main} hold. For any domain $D\subset \Omega$,  suppose that $u\in C^2(D)\cap C(\bar{D})$ satisfies $\mathcal{L}u\geq 0 \,(\leq 0 \,,\text{respectively})$ in $D$. Then the maximum (minimum, respectively) of $u$ in $\overline {D}$ is achieved on $\partial  D$, that is,
$$\sup_{D}u=\sup_{\partial D}u \,(\inf_{D}u=\inf_{\partial D}u \,,\text{respectively} ).$$
\end{lem}

 \begin{lem}\label{lem5.3}
Let the conditions of Theorem \ref{main} hold, and let domains
 $U\Subset V\Subset \Omega$.
For $x_0\in U$, we have that
\begin{align}
\inf_{U\backslash\{x_0\}} g_{ V}(\cdot,x_0)
=\inf_{\partial U} g_{ V}(\cdot,x_0)
\q \text{ and }\q
\sup_{ V\backslash U} g_{ V}(\cdot,x_0)
=\sup_{\partial U} g_{ V}(\cdot,x_0)\label{(5.20)}.
\end{align}
Similarly, for $x_0\in U$, we have that
\begin{align}
\inf_{U\backslash\{x_0\}} g_{ V}^*(\cdot,x_0)
=\inf_{\partial U} g_{ V}^*(\cdot,x_0)\q \text{ and }\q\sup_{ V\backslash U} g_{ V}^*(\cdot,x_0)
=\sup_{\partial U} g_{ V}^*(\cdot,x_0)\label{eq:adj(5.20)}.
\end{align}
\end{lem}
\begin{proof}
Since the proof of \eqref{eq:adj(5.20)} is same as that of \eqref{(5.20)}, we will only give the proof of \eqref{(5.20)}.
Fix $x_0\in U$.
Since $g_V(\cdot, x_0)$ is harmonic on $V\backslash\{x_0\}$ by Lemma \ref{Thm425}(2), applying Lemma \ref{prop4.2}, clearly we have the first equality of \eqref{(5.20)}.
For the second equality of \eqref{(5.20)},  it is enough to show that
\begin{align*}
\sup_{ V\backslash U} g_{ V}(\cdot,x_0)\le\sup_{\partial U} g_{ V}(\cdot,x_0)
\end{align*}
because of Lemma \ref{prop4.2}.
Indeed it is true since \cite[Theorem 7.3.3]{Pi95} implies that  for any $y\in V\backslash U$
\begin{align*}
g_V (y,x_0)=\E_y[g_V(X_{\tau_{\bar{U}}},x_0)1_{\{\tau_{U}<T_V\}}]
\le \sup_{\partial U}g_V(\cdot,x_0)\P_y(\tau_{U}<T_V)\le \sup_{\partial U}g_V(\cdot,x_0).
\end{align*}
\end{proof}

\subsection{Capacity and equilibrium measure}
In this subsection, we give some properties of the capacity for any $A, B\subset \Omega$ satisfying $\ass {S}$.
For brevity, we do not repeat this condition on $A, B$ and the conditions \eqref{gnrt} and (A1)--(A2) of $\cL$ in the subsection.

\begin{lem}
The capacity is invariant  with respect to the adjoint operator in the following sense:
	$$
	\ccap(A,B)=\ccap^*(A,B).
	$$
\end{lem}
\begin{proof}
	Denote by $h=h_{A,B}$ and $h^*=h^*_{A,B}$ the solutions to \eqref{pe} for $\cL$ and $\cL^*$, respectively.  Since $h=h^*=1$ on $\bar{A}$ and $h=h^*=0$ on $\bar B$, by the definition of capacities, we have
	$$
	\aligned
\ccap(A,B)&=-\langle h,\cL h\rangle_\mu
=\int_{\bar{A}}h (-\cL  h) d\mu=\int_{\bar{A}}h^* (-\cL  h) d\mu\nn\\
&=-\langle h^*,\cL  h\rangle_\mu=-\langle \cL^*  h^*, h\rangle_\mu = \int_{\bar{A}}(-\cL^* h^*)  h d\mu \nn\\
& =\int_{\bar{A}}(-\cL^* h^*)  h^* d\mu=-\langle  h^*, \cL^*h^*\rangle_\mu=\ccap^*(A,B).
\endaligned
	$$	
\end{proof}

In the following lemma, we introduce an equilibrium measure corresponding to the Poisson equation defined in \eqref{pe} which describes the capacity well.

\begin{lem}\label{lem6.5}
Let $h=h_{A, B}$ be the solution to \eqref{pe}  for $\cL$. Consider a measure
$$
\nu(dy):=\nu_{A, B}(dy):=-\a\nabla h\cdot {\bf n}_A \mu(dy)
$$
on $\partial A$ where ${\bf n}_A$ is the outward normal vector to $\partial A$. Then $\nu(dy)$ is a measure, which is called equilibrium measure, on $\partial A$ such that
	\begin{equation}\label{h-nu0}
	\ccap(A, B)=\nu(\partial A).
	\end{equation}
Also
	\begin{equation}
	\label{h-nu}
	h(x)=\int_{\partial A}\frac{g_{\overline B^c}(x,y)}{\mu(y)}\nu(dy) \qquad \text{for all}\ x\in \overline B^c\backslash \partial A
	\end{equation}
and in particular,
	\begin{equation}\label{h-nu1}
	\int_{\partial A}\frac{g_{\overline B^c}(x,y)}{\mu(y)}\nu(dy)=1 \qquad \text{for all}\ x\in A.
	\end{equation}
\end{lem}

\begin{proof}
Since $h\equiv1$ on $\bar{A}$ by the definitions of $h$, we have that
\begin{align}\label{eq:c}
\ccap(A, B)&=-\int_{\mathbb{R}^d}\,h\mathcal{L}hd\mu=-\int_{\bar{A}}\mu\nabla\cdot(\a\nabla h)dx-\int_{\bar{A}}\mu b\cdot\nabla hdx.
\end{align}
Using the divergence theorem with the fact that $h\equiv1$ on $\bar{A}$, we first observe that
\begin{align}\label{eq:c1}
\int_{\partial A}\mu \a\nabla h\cdot {\bf n}_A dx
&=\int_{\bar{A}}\nabla\cdot(\mu \a\nabla h)dx
=\int_{\bar{A}}\nabla\mu\cdot\a\nabla hdx+\int_{\bar{A}}\mu\nabla\cdot(\a\nabla h)dx\nn\\
&=\int_{\bar{A}}\nabla\cdot\big(h\a\nabla\mu\big)dx
-\int_{\bar{A}}\nabla\cdot\big(\a\nabla\mu\big)dx+\int_{\bar{A}}\mu\nabla\cdot(\a\nabla h)dx\nn\\
&=\int_{\partial A}\a\nabla\mu\cdot {\bf n}_A dx
-\int_{\partial A}\a\nabla\mu\cdot {\bf n}_Adx+\int_{\bar{A}}\mu\nabla\cdot(\a\nabla h)dx\nn\\
&=\int_{\bar{A}}\mu\nabla\cdot(\a\nabla h)dx.
\end{align}
Similarly we have that	
\begin{align*}
\int_{\partial A}\mu b\cdot{\bf n}_A dx
&=\int_{\bar{A}}\nabla\cdot(h\mu b)dx=\int_{\bar{A}}\nabla\cdot(\mu b)dx+\int_{\bar{A}}\mu b\cdot\nabla hdx\nn\\
&= \int_{\partial A}\mu b\cdot{\bf n}_A dx +\int_{\bar{A}}\mu b\cdot\nabla hdx,
\end{align*}
and therefore
\begin{align}\label{eq:c2}
\int_{\bar{A}}\mu b\cdot\nabla hdx=0.
\end{align}
By \eqref{eq:c}--\eqref{eq:c2}, we obtain the first assertion \eqref{h-nu0}.

Similarly using the divergence theorem, we can obtain that for any $x\in \overline B^c\backslash \partial A$,
\begin{equation}\label{h-nu:U1}
\aligned
&\quad\int_{\partial A}\frac{g_{\overline B^c}(x,\cdot)}{\mu(\cdot)}\nu(dy)=-\int_{\partial A}g_{\overline B^c}(x,\cdot)\a\nabla h\cdot{\bf n}_A dy=-\int_{\bar{A}} \nabla\cdot\big(g_{\overline B^c}(x,\cdot)\a\nabla h\big)dy\\
&=-\int_{\bar{A}}\nabla g_{\overline B^c}(x,\cdot)\cdot \a\nabla h dy-\int_{\bar{A}}g_{\overline B^c}(x,\cdot)\nabla\cdot (\a\nabla h) dy\\
&=-\int_{\bar{A}}\nabla\cdot\Big(h\a\nabla g_{\overline B^c}(x,\cdot)\Big)+h\nabla\cdot\Big(\a\nabla g_{\overline B^c}(x,\cdot)\Big)dy-\int_{\bar{A}}g_{\overline B^c}(x,\cdot)\nabla\cdot (\a\nabla h) dy\\
&=-\int_{\bar{A}}g_{\overline B^c}(x,\cdot)\nabla\cdot (\a\nabla h) dy.
\endaligned
\end{equation}
The last equality comes from the fact that $h=1$ on $\bar{A}$.
On the other hand, since $G_{\overline B^c}\cL h=-h$ from Lemma \ref{LG=GL}, and  $\mathcal{L}h=0$ on $(\overline {A}\cup \overline{B})^c$, we have that
\begin{equation}\label{h-nu:U2}
\aligned
h(x)&=-\int_{\overline B^c} g_{\overline B^c}(x,\cdot)\mathcal{L}hdy
=-\int_{\bar{A}} g_{\overline B^c}(x,\cdot)\big[\nabla\cdot(\a\nabla h)+b\cdot\nabla h\big]dy\\
&=-\int_{\bar{A}} g_{\overline B^c}(x,\cdot)\nabla\cdot(\a\nabla h)dy.
\endaligned
\end{equation}
The above last equality holds since
$$
\aligned
\int_{\bar{ A}} g_{\overline B^c}(x,y)b\cdot\nabla hdy=\int_{\bar{A}}\nabla\cdot\big(h g_{\overline B^c}(x,y)b\big)-h\nabla\cdot\big(g_{\overline B^c}(x,y)b\big)dy=0\ .
\endaligned
$$
Thus combining \eqref{h-nu:U1}--\eqref{h-nu:U2},  we obtain \eqref{h-nu}.
\end{proof}

We say that the elliptic Harnack inequality $\H$ holds on $\Omega$ if there exist constants $C\ge  1$ and $ \delta\in (0, 1)$ such that for any ball $B\subset \Omega$ of radius $R>0$, and for any harmonic and non-negative function $u$ in $B$, the following inequality holds:
	\begin{align*}
	\H: \qquad \esssup_{\delta B} u\le C \essinf_{\delta B} u.
	\end{align*}
The constants $C$ and $\delta$ are independent of the ball $B$ and the function $u$. Similarly, we denote by $\H^*$ the corresponding condition to the adjoint process $X^*$.

The following lemma is from \cite[Theorem 3.1]{S80}, which shows that the elliptic Harnack inequality holds for our processes $X$ and $X^*$.
\begin{lem}\label{Har}
For diffusions $X$ and $X^*$, $\H$ and $\H^*$ hold.
\end{lem}

From Lemma \ref{lem6.5}--\ref{Har}, we derive the relation between the capacity and the Green function.

\begin{lem}\label{lem7.2}
For any $x\in A$, we have
\be\label{(7.5)}
	\inf_{\partial A}g_{\overline B^c}^* (\cdot,x)\asymp \frac{\mu(x)}{\ccap(A,B)}\asymp \sup_{\partial A}  g_{\overline B^c}^* (\cdot,x).
	\de
\end{lem}
\begin{proof}
Our first claim is that for any $x\in A$,
\begin{equation}\label{lem71}
	\inf_{\partial A} g_{\overline B^c}^* (\cdot,x)\le \frac{\mu(x)}{\ccap (A,B)}\le \sup_{\partial A} g_{\overline B^c}^* (\cdot,x).
\end{equation}
Indeed, from \eqref{h-nu0},  we have that for all $x\in  \overline B^c$,
\begin{align*}
\inf_{\partial A} g_{\overline B^c}^*(\cdot,x)\frac{\ccap(A,B)}{\mu(x)}
\le \int_{\partial A}\frac{ g_{\overline B^c}^*(y,x)}{\mu(x)}\nu(dy)\le\sup_{\partial A} g_{\overline B^c}^*(\cdot,x)\frac{\ccap(A,B)}{\mu(x)}.
\end{align*}
By \eqref{dual:k} and \eqref{h-nu1}, since
$$\int_{\partial A}\frac{ g_{\overline B^c}^*(y,x)}{\mu(x)}\nu(dy)=\int_{\partial A}\frac{g_{ \overline B^c}(x,y)}{\mu(y)}\nu(dy)=1\q\q\text
{for all $x\in A$}\, ,$$
we obtain \eqref{lem71}.
On the other hand, since $ g_{\overline B^c}^*(\cdot,x)$ is harmonic for $\cL^*$ on $\overline B^c\backslash \{x\}$ by Lemma \ref{Thm425}(3), Lemma \ref{prop4.2} and the condition $\H^*$ imply that
\be
\HG^*: \q\q \q\q\sup_{\partial A}  g_{\overline B^c}^* ( \cdot,x)\le c \inf_{\partial A}g_{\overline B^c}^*(\cdot,x)\nn
\de
for some constant $c>0$ which is independent of $A$ and $B$. Combining \eqref{lem71} and $\HG^*$ it gives the desired result.
\end{proof}

\section{Proof of Theorem \ref{main}}\label{proof}

\subsection{Implication $\CC\Rightarrow\CG$}

In this subsection, we will prove $\CC\Rightarrow\CG$ of our main Theorem \ref{main}.

\begin{lem}\label{Lem:ng}
	For any $x_0\in \Omega$ and $R>0$, define $B:=B(x_0, R)$ such that $B_n:=K^n B\subset\Omega$ where $K\ge 1$ is the constant in $\CC$.
	Then there is a constant $c>0$ such that for any $n\ge m\ge 0$
	\begin{align*}
	\sup_{\partial B_m} g_{B_n}^*(\cdot, x_0)\le  c \mu(x_0)\sum_{k=m}^{n-1} \left(\ccap(B_k, \overline B^c_{k+1}) \right)^{-1}.
	\end{align*}
\end{lem}
\begin{proof}
	For each $k\ge 0$, we will show that for any $y\in \Omega\backslash\{x_0\}$,
	\begin{align}\label{eq:gBcl}
	g_{B_{k+1}}^*(y, x_0)-	g_{B_{k}}^*(y, x_0)\le \sup_{B_{k+1}\backslash B_k}	g_{B_{k+1}}^*(\cdot,  x_0).
	\end{align}
	If \eqref{eq:gBcl} holds true, combining \eqref{(5.20)} and \eqref{(7.5)}, we have that
	$$
	g_{B_{k+1}}^*(y, x_0)-	g_{B_{k}}^*(y, x_0)\le \sup_{B_{k+1}\backslash B_k}	g_{B_{k+1}}^*(\cdot, x_0)=\sup_{\partial B_k} g_{B_{k+1}}^*(\cdot,  x_0)\le  c\mu(x_0) \left(\ccap(B_k, \overline B^c_{k+1}) \right)^{-1}$$
for some constant $c>0$.
Then by adding $k$ from $m+1$ to $n-1$,
we obtain that for any $y\in \Omega\backslash\{x_0\}$,
	\begin{align*}
	g_{B_{n}}^*(y, x_0)-g_{B_{m+1}}^*(y, x_0)
	\le  c\mu(x_0) \sum_{k=m+1}^{n-1}\left(\ccap(B_k, \overline B^c_{k+1}) \right)^{-1}.
	\end{align*}
Since \eqref{(7.5)} implies
$$\sup_{\partial B_{m}} g_{B_{m+1}}^*(\cdot, x_0)\le c\mu(x_0) \left(\ccap(B_m,  \overline B^c_{m+1}) \right)^{-1},$$
combining above two inequalities, we obtain our assertion.

Now we prove \eqref{eq:gBcl}. Since $g_{B_k}^*(y,x_0)=0$ for $y\in B_k^c$ and $g_{B_{k+1}}^*(y,x_0)=0$ for $y\in B_{k+1}^c$, it is sufficient to show that \eqref{eq:gBcl} holds for $y\in B_k\backslash\{x_0\}$. For any $f\in  \mathcal {B}^{\alpha}(B_k)\cap C^2(B_k)$ which is $f\geq0$ in $\Omega$ and $f=0$ in $B_k^c$, let $u:=G_{B_k+1}^* f-G_{B_k}^* f$.
Then $u$ is harmonic in $B_{k}$ by Lemma \ref{LG=GL}, so Lemma \ref{prop4.2} yields that  for any $B'\Subset B_k$,
\begin{align*}
	\sup_{B_k} u=\sup_{\partial B_k} u\leq \sup_{B_{k+1}\backslash B'} u \le \sup_{B_{k+1}\backslash B'} G_{B_{k+1}}^* f.
\end{align*}
Therefore for any $B'\Subset B_k$,
\begin{align}\label{eq:supG1}
\sup_{B_{k+1}}\left(G_{B_k+1}^* f-G_{B_k}^* f\right)	&=\Big(\sup_{B_k} u\Big)\vee \Big(\sup_{B_{k+1}\backslash B_k} u\Big)\\
&\le \Big(\sup_{B_{k+1}\backslash B'} G_{B_{k+1}}^* f\Big)\vee \Big(\sup_{B_{k+1}\backslash B_k}G_{B_{k+1}}^* f\Big)=\sup_{B_{k+1}\backslash B'} G_{B_{k+1}}^* f.\nn
\end{align}
	Consider a ball  $B'\Subset B_k$ containing $x_0$ and a positive sequence $\{\eps_{n}\}$ which is $\eps_{n}\rightarrow 0$ as $n\to \infty$. Choose functions $f_{n, x_0}\in \mathcal {B}^{\alpha}(B_k)\cap C^2(B_k)$ supported in $B(x_0, \eps_{n})\subset B_k$ and satisfying that $f_{n,x_0}\geq 0$, $\int_\Omega f_{n,x_0}dx=1$ and $f_{n, x_0}$ is weakly converge to the Dirac delta function $\delta_{x_0}$ in $C(\Omega)$ as $n\to \infty$.
	Applying $f_{n,x_0}$ to \eqref{eq:supG1}, and letting $B'\uparrow B_k$,  we have that for any $y\in B_k\backslash\{x_0\}$
\begin{align}\label{eq:supG3}
	G_{B_k+1}^* f_{n, x_0}(y)-G_{B_k}^* f_{n, x_0}(y)\le\sup_{B_{k+1}\backslash B_k} G_{B_{k+1}}^* f_{n, x_0}= \sup_{\partial B_k}G_{B_{k+1}}^*f_{n,x_0}.
	\end{align}
The above last equality comes from Lemma \ref{prop4.2} since $G_{B_{k+1}}^*f_{n,x_0}$ is harmonic in $B_{k+1}\backslash B_k$ and $G_{B_{k+1}}^*f_{n,x_0}=0$ on $\partial B_{k+1}$. On the other hand, since for any $\varphi\in C_0(\Omega)$
	$$
	\aligned
	\langle G_{B_k}^* f_{n,x_0},\varphi \rangle_{\mu}
	=\langle f_{n,x_0},G_{B_k} \varphi\rangle_{\mu}
	&=\int_{B(x_0, \eps_{n})}f_{n,x_0}(y)\int_{B_k} \varphi(z) g_{B_k}(y, z) dz \,\mu(y)dy\\
	&=\int_{B_k} \varphi(z) \int_{B(x_0, \eps_{n})} f_{n,x_0}(y)g_{B_k}^*(z, y)\mu(z) dydz\nn\\	
&\rightarrow\, \int_{B_k} \varphi(z)g_{B_k}^*(z, x_0)\mu(z) dz=\langle g_{B_k}^*(\cdot,x_0),\varphi \rangle_{\mu}
	\endaligned
	$$
	as $n\to \infty$,
	we have that
	\begin{align*}
	\lim_{n \rightarrow \infty}G_{B_k}^* f_{n,x_0}(y)=g_{B_k}^*(y,x_0) \quad\text{ for }y\in \Omega\backslash\{x_0\}.
	\end{align*}
Similarly, we have that $\lim_{n \rightarrow \infty}G_{B_{k+1}}^* f_{n,x_0}(y)=g_{B_{k+1}}^*(y,x_0)$ for $y\in \Omega\backslash\{x_0\}$. Thus letting $n$ goes to infinity in \eqref{eq:supG3}, we obtain
\begin{equation}\label{limsup}
g_{B_{k+1}}^*(y,x_0)-g_{B_k}^*(y,x_0)\le \lim_{n\rightarrow\infty}\sup_{\partial B_k}G_{B_{k+1}}^*f_{n,x_0},\quad \text{for }y\in B_k\backslash\{x_0\}.
\end{equation}
Since $G_{B_{k+1}}^*f_{n,x_0}(\cdot)$, $n\geq 0$, $g_{B_{k+1}}^*(\cdot,x_0)$ is continuous on $\partial B_k$, and $\partial B_k$ is compact,  Dini's theorem with \eqref{eq:adj(5.20)} implies that
$$
\lim_{n\rightarrow\infty}\sup_{\partial B_k}G_{B_{k+1}}^*f_{n,x_0}=\sup_{\partial B_k}\lim_{n\rightarrow\infty}G_{B_{k+1}}^*f_{n,x_0}=\sup_{\partial B_k}g_{B_{k+1}}^*(\cdot,x_0)=\sup_{B_{k+1}\backslash B_k}g_{B_{k+1}}^*(\cdot,x_0).
$$
Combining this with \eqref{limsup}, we complete the proof of \eqref{eq:gBcl}.
\end{proof}

{\bf proof of $\CC\Rightarrow$ $\CG$:}
Since $\CG$ and $\CGad$ are equivalent by \eqref{g_g*}, it is sufficient to show that $\CC\Rightarrow\CGad$.
Let $K\ge 1$ be the constant in $\CC$.
For any $x_0\in \Omega$ and $R>0$, consider a ball $B:=B(x_0, R)\subset \Omega$.
For any $y\in  K^{-1} B\backslash \{x_0\}$, let  $d_0:=|x_0- y|$ and  $B_0:=B(x_0, d_0)$. Then $0<d_0< K^{-1}R$  and $B_0\subset KB_0\subset B\subset \Omega$. Using \eqref{(7.5)} and $\CC$, since $\mu\in C^1(\Omega)$, we have that
\begin{align*}
\inf_{\partial  B_0}g_{KB_0}^*(\cdot, x_0)
\asymp\frac{\mu(x_0)}{\ccap(B_0, \overline {KB_0}^c)} \asymp\frac{\mu(x_0) d_0^2}{\mu(B_0)}\asymp d_0^{2-d}.
\end{align*}
Applying \eqref{eq:adj(5.20)} in Lemma \ref{lem5.3}, we have that for any $y\in K^{-1}B\backslash\{x_0\}$,
\begin{align*}
g_B^*(y, x_0)\ge \inf_{\overline {B_0}\backslash\{x_0\}} g_{KB_0}^*(\cdot, x_0) =\inf_{\partial  B_0}g_{KB_0}^*(\cdot, x_0)\ge c \, d_0^{2-d}\,
\end{align*}
for some $c>0$, which yields $\CGgead$.

To prove $\CGlead$, we first note that
\begin{align}\label{eq:CGm2}
g_B^*(y, x_0)&=g_B^*(y, x_0)\1_{K^{-1}B}(y)+g_B^*(y, x_0)\1_{B\backslash K^{-1}B}(y).
\end{align}
For $y\in B\backslash K^{-1}B$, by the similar way to the proof of $\CGgead$, \eqref{eq:adj(5.20)},  \eqref{(7.5)}, $\CC$ and the fact that $\mu\in C^1(\Omega)$ with $K^{-1}B\subset B\subset D$ yield that
\begin{align}
\sup_{B\backslash K^{-1}B}  g_{B}^*(\cdot, x_0)
&=\sup_{\partial  (K^{-1}B)}g_{B}^*(\cdot, x_0)\asymp \frac{\mu(x_0)}{\ccap(K^{-1}B,  \overline B^c)}\nn\\
&\asymp \frac{\mu(x_0)(K^{-1}R)^2}{\mu(K^{-1}B)}\asymp(K^{-1}R)^{2-d}= c R^{2-d}\le c d_0^{2-d}\label{eq:CG2}
\end{align}
for some $c>0$.
If $y\in K^{-1}B$, that is, $0<d_0< K^{-1}R$,
we choose an integer $n>1$ such that $K^{-n-1}R\le d_0<K^{-n}R$, and define $B_i:=B(x_0, K^{-i}R)$ so that $B_{i+1}\subset B_i$. Then $\CC$ with $K^{-1}B_i=B_{i+1}\subset B_i$ for $i=0, \ldots, n$ implies that
$$\ccap(B_{i+1},\overline B_i^c)^{-1}\nn\\
\le c\mu(x_0)\frac{(K^{-(i+1)}R)^2}{\mu(B_{i+1})}.
$$
Using this with  \eqref{eq:adj(5.20)} and Lemma \ref{Lem:ng}, we have that
\begin{align*}
\sup_{B\backslash B_{n+1}}g_B^*(\cdot, x_0)
&=\sup_{\partial B_{n+1}}g_B^*(\cdot, x_0)
\le c \mu(x_0)\sum_{i=0}^n \ccap(B_{i+1},\overline B_i^c)^{-1}\nn\\
&\le c\mu(x_0)\sum_{i=0}^n \frac{(K^{-(i+1)}R)^2}{\mu(B_{i+1})}.
\end{align*}
Since $\mu\in C^1(\Omega)$, $\mu(B_{i+1})\asymp \mu(x_0)(K^{-(i+1)}R)^d$ and therefore
\begin{align}\label{eq:CG3}
\sup_{B\backslash B_{n+1}}g_B^*(\cdot, x_0)&\le
 c(K^{-n}R)^{2-d}\sum_{i=0}^n K^{(n-i-1)(2-d)}\nn\\
&\le c d_0^{2-d}\sum_{i=0}^n K^{(n-i-1)(2-d)}\le c d_0^{2-d}
\end{align}
for some $c>0$.
Hence, we have $\CGlead$ by
\eqref{eq:CGm2}--\eqref{eq:CG3}.
\qed

\subsection{Implication $\CG\Rightarrow \CEad$}

In this subsection, we will prove $\CG\Rightarrow\CEad$ of our main Theorem \ref{main}.

\noindent
{\bf Proof of $\CG\Rightarrow\CEad$.}
It is enough to show that  $\CGad\Rightarrow\CEad$ by \eqref{g_g*}.  Fix a ball $B=B(x_0, R)$.
Since $B\subset B(y,2R)$ for $y\in B$,  \eqref{g_e} and $\CGlead$ imply that for any $x\in B$,
\begin{align}\label{eq:GE1}
\E_x[T^*_B]
=\int_B g_B^*(x, y)dy
&\leq \int_B g_{B(y,2R)}^*(x, y)dy\leq C_G^*\int_B |x-y|^{2-d}dy\nn\\
&\leq  c \int_0^{ 2R} s^{2-d} s^{d-1} ds= c R^2
\end{align}
for some $c>0$ and this give the proof of $\CElead$.
To prove the lower bounds of $\E_x[T^*_B]$, that is $\CEgead$,
choose $\delta:=K^{-1}\in (0, 1)$ where $K>1$ is the constant in $\CGad$.
For fixed point $y\in \delta_1 B$, let $B':=B(y,(1-\delta)R)$ then $B'\subset B$. Using the similar argument to the upper bound estimates, \eqref{g_e} and $\CGgead$  imply that for any $x\in \delta B$,
\begin{align}\label{eq:GE2}
\E_x[T_B^*]
=\int_B g_B^*(x,y)dy
&\geq
\int_{B} g_{B'}^*(x,y)dy\geq (C_G^*)^{-1} \int_{B}|x-y|^{2-d}dy\nn\\
&\geq c \int_0^{ 2\delta_1R}s^{2-d}s^{d-1}dy=c R^2
\end{align}
for some $c>0$.
Therefore we have $\CEad$  by \eqref{eq:GE1} and \eqref{eq:GE2}.
\qed

\subsection{Implication $\CEad\Rightarrow \CC$}
To prove $\CEad\Rightarrow\CC$ in Theorem \ref{main}, we introduce an auxiliary lemma between the capacity and the mean exit time firstly.

\begin{lem}\label{lem8.3-8.4}
	Consider sets $A, B$ satisfying the $\ass{S}$. Then we have that
	$$	
	\frac{\mu(A)\big(\inf_{x\in\partial A}\mathbb{E}_x[T^*_{\bar B^c}]\big)^2}{\mu(\bar B^c)^2
		\sup_{x\in \bar{B}^c}\mathbb{E}_x[T^*_{\bar B^c}]}\leq \ccap (A,B)^{-1}\leq\frac{\sup_{x\in \bar B^c}\mathbb{E}_x[T^*_{\bar B^c}]}{\mu(A)}.
	$$
\end{lem}
\begin{proof}
For any sets $A, B$  satisfying the \ass{S}, let $h(x):=h_{A,B}(x)=\P_x(\tau_A<\tau_{B})$ be the solution of \eqref{pe}.
By Lemma \ref{lem6.5} we recall that
$$\ccap(A, B)=\nu(\partial A)\q\text{ and } \q
	h(x)=\int_{\partial A}\frac{g_{\bar B^c}(x,y)}{\mu(y)}\nu(dy)\ \q\text{ for } x\in \bar B^c\backslash\partial A.
	$$
From Fubini's theorem with \eqref{g_e}, \eqref{dual:k} and \eqref{h-nu0}, we have that
	\begin{align}\label{hnuup}
	\int_{\bar B^c} h(x)\mu(dx)&=\int_{\partial A}\int_{\bar B^c} g_{\bar B^c}(x,y)\frac{\mu(x)}{\mu(y)}dx\nu(dy)\nn\\
	&=\int_{\partial A}\int_{\bar B^c} g_{\bar B^c}^*(y,x)dx\,\nu(dy)=\int_{\partial A}\mathbb{E}_y[T^*_{\bar B^c}]\,\nu(dy)\nn\\
	&\leq \nu(\partial A)\sup_{x\in\partial A}\mathbb{E}_x[T^*_{\bar B^c}]\leq \ccap(A,B)\sup_{x\in \bar B^c}\mathbb{E}_x[T^*_{\bar B^c}].
	\end{align}
Since $h(x)=1$ in $A$,  we have that $\mu(A)\le \int_{\bar B^c} h(x)\mu(dx)$ and therefore
\begin{equation}\label{cap:upb}
\frac{1}{\ccap(A,B)}\leq \frac{\sup_{x\in \bar B^c}\mathbb{E}_x[T^*_{\bar B^c}]}{\int_{\bar B^c} h(x)\mu(dx)}\leq \frac{\sup_{x\in \bar B^c}\mathbb{E}_x[T^*_{\bar B^c}]}{\mu(A)}.
\end{equation}
Similar way to obtaining \eqref{hnuup}, we have the lower bound as follows:
$$\int_{\bar B^c} h(x)\mu(dx)= \int_{\partial A}\mathbb{E}_y[T^*_{\bar B^c}]\nu(dy)
\geq \nu(\partial A)\inf_{x\in\partial A}\mathbb{E}_x[T^*_{\bar B^c}]=\ccap(A,B)\inf_{x\in\partial A}\mathbb{E}_x[T^*_{\bar B^c}].
$$
Combining the above two inequalities with the Cauchy-Schwarz inequality, we have that
$$
\aligned
\frac{\ccap(A, B)}{\int_{\bar{B}^c}h^2(x)\mu(dx)}&\leq\frac{\mu(\bar B^c)^2\ccap(A,B)}{\big(\int_{\bar B^c} h(x)\mu(dx\big)^2}\\
	&\leq \frac{\mu(\bar B^c)^2}{\ccap(A,B)\big(\inf_{x\in\partial A}\mathbb{E}_x[T^*_{\bar B^c}]\big)^2}\leq\frac{\mu(\bar B^c)^2\sup_{x\in \bar B^c}\mathbb{E}_x[T^*_{\bar B^c}]}{\mu(A)\big(\inf_{x\in\partial A}\mathbb{E}_x[T^*_{\bar B^c}]\big)^2}.
	\endaligned
	$$
	Since $\int_{\bar{B}^c}h^2(x)\mu(dx)\le 1$, we have that
	\begin{equation}\label{cap:lowb}
	\frac{1}{\ccap(A,B)}
	\geq \frac{\mu(A)\big(\inf_{x\in\partial A}\mathbb{E}_x[T^*_{\bar B^c}]\big)^2}{\mu(\bar B^c)^2\sup_{x\in \bar B^c}\mathbb{E}_x[T^*_{\bar B^c}]}.
	\end{equation}
	Therefore, \eqref{cap:upb}--\eqref{cap:lowb} give the desired result.
\end{proof}

{\bf Proof of ${\CEad\Rightarrow \CC}$.}
Let $B\subset\Omega$ be a ball with radius $R>0$. For any $\delta\in(0, 1)$, Lemma \ref{lem8.3-8.4} and $\CElead$ imply
\begin{align}\label{eq:EC1}
\ccap(\delta B,\overline B^c)^{-1}\leq \frac{\sup_{x\in B}\mathbb{E}_x[T^*_{B}]}{\mu(\delta B)}\leq c\frac{R^2}{\mu(\delta B)}
\end{align}
for some constant $c>0$.
On the other hand, to prove $\CCge$, consider the  constant $\delta$ in $\CEgead$.  For any $0<\delta<\delta_1$,
since $x\mapsto\mathbb{E}_x [T^*_B]$ is continuous in $B$,
$\CEgead$ implies that
$$
\inf_{x\in\partial (\delta B)}\mathbb{E}_x[T^*_B]\geq\inf_{x\in \delta B}\mathbb{E}_x[T^*_B]\geq c^{-1}R^2.
$$
Combining the above inequality with Lemma \ref{lem8.3-8.4} and $\CElead$, we conclude that
\begin{align}\label{eq:EC2}
\ccap(\delta B,\overline B^c)^{-1}&\geq \frac{\mu(\delta B)\big(\inf_{x\in\partial (\delta B)}\mathbb{E}_x[T^*_B]\big)^2}{\mu(B)^2\sup_{x\in B}\mathbb{E}_x[T^*_B]}\geq  \frac{\mu(\delta B)[c^{-1}R^2]^2}{\mu(B)^2 cR^2}\geq c\frac{R^2}{\mu(\delta B)}
\end{align}
for some $c>0$. The last inequality holds since $\mu\in C^1(\Omega)$.
Therefore, we obtain $\CCge$ by \eqref{eq:EC1} and $\CCle$ by \eqref{eq:EC2} with $K=\delta^{-1}$.
\qed

{\bf Acknowledgement} Lu-Jing Huang acknowledges support from NSFC No. 11771047, 11901096 and Probability and Statistics: Theory and Application (IRTL1704).

\bibliographystyle{plain}
\bibliography{CE}

\begin{thebibliography}{10}

\bibitem{BL12}
J.~Beltr\'an and C.~Landim.
\newblock Metastability of reversible condensed zero range processes on a
  finite set.
\newblock {\em Probab. Theory Relat. Fields}, 152:781--807, 2012.

\bibitem{BL12a}
J.~Beltr\'an and C.~Landim.
\newblock Tunneling and metastability of continuous time markov chains ii, the
  nonreversible case.
\newblock {\em J. Stat. Phys.}, 149:598--618, 2012.

\bibitem{BGL07}
I.~Benjamini, O.~Gurel-Gurevich, and R.~Lyons.
\newblock Recurrence of random walk traces.
\newblock {\em Ann. Probab.}, 35:732--738, 2007.

\bibitem{B06}
A.~Bovier.
\newblock {\em {Metastability: a potential theoretic approach. }}.
\newblock International Congress of Mathematicians. Vol. III, 499–518, Eur.
  Math. Soc., Zürich, 2006.

\bibitem{BEGK02}
A.~Bovier, M.~Eckhoff, V.~Gayrard, and M.~Klein.
\newblock Metastability and low-lying spectra in reversible markov chains.
\newblock {\em Comm. Math. Phys.}, 228:219--255, 2002.

\bibitem{BEGK04}
A.~Bovier, M.~Eckhoff, V.~Gayrard, and M.~Klein.
\newblock Metastability in reversible diffusion processes 1. sharp estimates
  for capacities and exit times.
\newblock {\em J. Eur. Math. Soc}, 6:399--424, 2004.

\bibitem{C04}
M.-F. Chen.
\newblock {\em From Markov chains to non-equlilibrium particle systems}.
\newblock World Scientific Publishing Co. Pte. Ltd., 2004.

\bibitem{DS84}
P.~G. Doyle and J.~L. Doyle.
\newblock {\em Random Walks and Electric Networks}.
\newblock The Carus Math. Monographs 22, Math. Association of America, 1984.

\bibitem{GL14}
A.~Gaudilli\`ere and C.~Landim.
\newblock A dirichlet principle for non reversible markov chains and some
  recurrence theorems.
\newblock {\em Probab. Theory Relat. Fields}, 158:55--89, 2014.

\bibitem{GT01}
D.~Gilbarg and N.~S. Trudinger.
\newblock {\em Elliptic partial differential equations of second order}.
\newblock Springer-Verlag Berlin Heidelberg, 2001.

\bibitem{GH14}
A.~Grigor\'yan and J.~Hu.
\newblock Heat kernels and green functions on metric measure spaces.
\newblock {\em Canadian Journal of Mathematics}, 66(3):641--699, 2014.

\bibitem{L14}
C.~Landim.
\newblock Metastability for a non-reversible dynamics: The evolution of the
  condensate in totally asymmetric zero range processes.
\newblock {\em Commun. Math. Phys}, 330:1--32, 2014.

\bibitem{LMS19}
C.~Landim, M.~Mariani, and I.~Seo.
\newblock Dirichlet's and thomson's principles for non-selfadjoint elliptic
  operators with application to non-reversible metastable diffusion processes.
\newblock {\em Archive for Rational Mechanics and Analysis}, 231(2):887--938,
  2019.

\bibitem{Pi95}
R.~G. Pinsky.
\newblock {\em Positive harmonic functions and diffusions}.
\newblock Cambridge University Press, 1995.

\bibitem{S80}
M.~V. Safonov.
\newblock Harnack's inequality for elliptic equations and h\"older property of
  their solutions.
\newblock {\em Zap. Nauchn. Sem. Leningrad. Otdel. Mat. Inst. Steklov. (LOMI)},
  96:272--287, 1980.

\bibitem{So94}
P.~M. Soardi.
\newblock {\em Potential theory on infinite networks}.
\newblock Springer, Berlin, 1994.

\bibitem{Wo00}
W.~Woess.
\newblock {\em Random Walks on Infinite Graphs and Groups}.
\newblock Cambridge University Press, Cambridge, 2000.

\end{thebibliography}

\vskip 0.3truein

{\bf Lu-Jing Huang}
\vskip -.1truein
College of Mathematics and Informatics, Fujian Normal University, Fuzhou, 350007, P.R. China
\vskip -.1truein
E-mail: \texttt{huanglj@fjnu.edu.cn}

\medskip

{\bf Kyung-Youn Kim*}
\vskip -.1truein
Institute of Mathematics, Academia Sinica, Taipei, 11529,
Taiwan
\vskip -.1truein
E-mail: \texttt{kykim@gate.sinica.edu.tw}
\end{document}